\documentclass[11pt,a4paper]{article}

\usepackage[T1]{fontenc}
\usepackage[utf8]{inputenc}
\usepackage{amsmath,amsthm,amssymb}
\usepackage{enumitem}
\usepackage[colorlinks=true,linkcolor=blue,citecolor=blue,urlcolor=blue]{hyperref}
\providecommand{\colonequals}{\mathrel{:=}}
\newcommand{\Nil}{\operatorname{Nil}}
\newcommand{\Int}{\operatorname{Int}}
\newcommand{\Jac}{\operatorname{Jac}}

\newcommand{\irad}[1]{\sqrt[\mathrm{i}]{#1}}
\newcommand{\grad}[1]{\sqrt[\mathrm{g}]{#1}}
\newcommand{\jrad}[1]{\sqrt[\mathrm{j}]{#1}}
\newcommand{\frad}[1]{\sqrt[\mathrm{f}]{#1}}

\newcommand{\Mat}{\operatorname{M}}

\theoremstyle{plain}
\newtheorem{theorem}{Theorem}[section]
\newtheorem{corollary}[theorem]{Corollary}

\newtheorem{lemma}[theorem]{Lemma}
\newtheorem{proposition}[theorem]{Proposition}

\newtheorem*{conjecture*}{Conjecture}

\newtheorem{definition}{Definition}[section]

\newtheorem{example}{Example}[section]
\newtheorem{remark}{Remark}[section]
\numberwithin{equation}{section}

\title{Integral Dependence over One-Sided Ideals, Part II: G-Left Ideals and the K\"othe Problem} 
\author{Masood Aryapoor\\
	\tiny{\textit{Department of Business and Mathematics}}\\
	\tiny{\textit{M\"{a}lardalen  University}}\\
	\tiny{\textit{Hamngatan 15, 632 17, Eskilstuna, 
			Sweden
	}}
}
\date{}
\begin{document}
	\maketitle
	\begin{abstract} 
		We study integral dependence over left ideals using G-left ideals, namely left ideals maximal with respect to the failure of integrality of a fixed element. We establish several characterizations of G-left ideals, and show that their intersection is precisely the upper nilradical. Motivated by the  K\"othe conjecture, we introduce the class of left KI-rings as rings in which sums of left ideals integral over a fixed left ideal remain integral over that ideal. It is shown that left artinian rings,  PI-rings, algebraic algebras, and algebras of dimension smaller than the cardinality of the ground field are left KI-rings. 
	\end{abstract}
	\section{Introduction}
	In Part I~\cite{aryapoor2026integral}, we introduced the notion of integral dependence over a one-sided ideal, which  provides a way to extend the concept of nilness to the setting of one-sided ideals. Following Part I~\cite{aryapoor2026integral}, we say that an element \(a\in R\) is integral over a left ideal \(I\) if
	\[
	a^n+r_{n-1}a^{n-1}+\cdots+r_0=0\qquad(r_i\in I).
	\]
	This notion is the specialization to left ideals of the approaches to noncommutative integrality developed by Schelter and Quinn~\cite{schelter1976integral,schelter1976integral-errata,quinn1989integrality}. 
	
	The central objects of this paper are \emph{G-left ideals}. Such an ideal \(P\) admits a \emph{witness} \(a\) that is not integral over \(P\), but becomes integral over every strictly larger left ideal. This generalizes the classical characterization of a commutative G-domain in terms of a nonzero element that becomes nilpotent in every proper quotient. We show that the assignment
	\[
	P\longmapsto P[t]+R[t](1-at)
	\]
	identifies G-left ideals with witness \(a\) with maximal left ideals of the polynomial ring \(R[t]\) that contain \(1-at\) (Proposition~\ref{prp:char-G-left-R[t]}). This can be viewed as a one-sided form of the mechanism underlying the Rabinowitsch trick.
	The intersection of the G-left ideals containing a left ideal \(I\) defines its \emph{G-radical} \(\grad{I}\). We show that this radical is integral over \(I\) (Proposition~\ref{prp:grad-semi-int}), and that it is the largest left ideal whose addition to any left ideal containing \(I\) does not affect integral dependence (Proposition~\ref{prp:negligible}). At the zero ideal, the G-radical coincides with the classical upper nilradical (Theorem~\ref{thm:G-radical.2-sided}).
	
	Motivated by the K\"othe conjecture, we introduce the notion of a \emph{left KI-ring}, namely a ring in which sums of left ideals that are integral over a fixed left ideal remain integral over that ideal. We establish several characterizations of KI-rings parallel to the classical equivalent formulations of the K\"othe conjecture (Proposition~\ref{prp:gen-kothe-1st-equiv} and Proposition~\ref{prp:gen-kothe-2nd-equiv}). We show that left artinian rings, PI-rings, algebraic algebras,  and algebras of dimension smaller than the cardinality of the ground field are all left KI-rings (Theorem~\ref{thm:positive} and Theorem~\ref{thm:PI}).

	Here is an outline of the paper. Section~\ref{sec:back-pre} recalls the necessary material from Part I and provides some preliminaries. Section~\ref{sec:G-left} studies G-left ideals and G-radicals. Section~\ref{sec:kothe} introduces the concept of a left KI-ring, establishes several characterizations of this class, and provides illustrative examples. 
	
	All rings are associative and unital, and ring extensions share the identity. Ideals called two-sided are explicitly identified as such; otherwise our constructions concern left ideals. Unless stated otherwise, \(R\) denotes a unital ring and all ideals under consideration are ideals of \(R\).

	\section{Background and Preliminaries}\label{sec:back-pre}
	This part presents the preliminaries and recalls the results of Part I~\cite{aryapoor2026integral} needed below. The polynomial and companion-matrix constructions are developed here because they are needed for the G-left ideal correspondences in Section~\ref{sec:G-left}. 
	
	\subsection{Saturated left ideals}
	
	For a left ideal \(I\) of \(R\) and an element \(a\in R\), set
	\[
	(I:a) \colonequals \{\, r\in R \mid ra\in I \,\}.
	\]
	A left ideal \(I\subseteq R\) is called \(a\)-closed if  
	\(I\subseteq (I:a)\).  The smallest \emph{\(a\)-closed}
	left ideal that contains \(I\), called the \emph{\(a\)-closure} of \(I\),  can be expressed as 
	\[
	Ia^\infty \colonequals \sum_{n\ge 0} Ia^n,
	\]
	where \(Ia^i\) denotes the set \(\{ra^i\mid r\in I\}\). A left ideal \(I\subseteq R\) is \(a\)-saturated if  
	\((I:a)\subseteq I\).  The smallest
	\(a\)-saturated left ideal containing \(I\) is called the \emph{\(a\)-saturation} of \(I\), and is described in the following proposition recalled from Part I~\cite{aryapoor2026integral}. 
	\begin{proposition}\label{prp:a-saturation}
		For any left ideal \(I\subseteq R\), the \(a\)-saturation of \(I\) coincides with 
		\[
		(I:a^\infty) \colonequals \bigcup_{n\ge 0} (\,I + Ia + \cdots + Ia^{n} : a^{n}\,).
		\]
	\end{proposition}
	A left ideal \(I\subseteq R\) is \(a\)-complete if   \((I:a) = I\). The smallest \(a\)-complete
	left ideal that contains \(I\) is called the \emph{\(a\)-completion} of \(I\). 
	
	The following example will be used throughout the paper. 
	\begin{example}\label{exm:matrix-rings-clo-sat-com}
		Let \(D\) be a division ring and \(R=\Mat_n(D)\). A left ideal is the set of matrices annihilating a right \(D\)-subspace \(V\subseteq D^n\). For \(RB\), this subspace is \(V_B=\ker B\). Thus \(RA\subseteq RB\) exactly when \(V_B\subseteq V_A\), and \((RB:A)\) corresponds to \(A(V_B)\). Consequently \(RB\) is \(A\)-closed when \(A(V_B)\subseteq V_B\), and is \(A\)-saturated when \(V_B\subseteq A(V_B)\). Finite dimensionality makes the latter condition equivalent to \(A(V_B)=V_B\), so saturation and completion agree in this example. The closure corresponds to the largest \(A\)-invariant subspace of \(V_B\); the saturation corresponds to the largest subspace \(W\subseteq V_B\) on which \(A\) maps \(W\) onto itself.
	\end{example}
	Saturation can be studied using the polynomial ring \(R[t]\), where \(t\) is a central variable. More precisely, we have the following result recalled from Part I~\cite{aryapoor2026integral}.  
	\begin{proposition}\label{prp:a-saturated-R[t]}
		For every left ideal \(I\subseteq R\) and \(a\in R\), we have 
		\[
		(I:a^\infty) = \bigl( I[t] + R[t](1-at)\bigr) \cap R.
		\]
		In particular, \(I\) is \(a\)-saturated iff
		\(
		I = \bigl( I[t] + R[t](1-at)\bigr) \cap R.
		\)
	\end{proposition}
	
	Next, we establish a bijective correspondence between the set of \(a\)-complete left ideals of \(R\) and a certain class of left ideals of the polynomial ring \(R[t]\). We need some preparatory results. 
	\begin{proposition}\label{prp:division-1-at}
		For \(f(t) = r_0+r_1t+\cdots+r_nt^n\in R[t]\) of degree \(n\) and \(a\in R\), there exists a unique \(g\in R[t]\) with degree \(<n\) such that
		\[
		f(t) =  (\sum_{i=0}^nr_ia^{n-i})t^n + g(t)(1-at).
		\]
	\end{proposition}
	\begin{proof}
		Dividing the polynomial \(f(t^{-1})t^n\in R[t]\) by \(t-a\), we obtain
		\[
		f(t^{-1})t^n = h(t) (t-a) + \sum_{i=0}^nr_ia^{n-i},
		\]
		where \(h(t)\in R[t]\) with degree \(<n\). Replacing \(t\) by \(t^{-1}\) and multiplying by \(t^n\) yields
		\[
		f(t) = (h(t^{-1})t^{n-1})(1-at) + (\sum_{i=0}^nr_ia^{n-i})t^n.
		\]
		Since the degree of \(g(t) = h(t^{-1})t^{\,n-1}\in R[t]\) is less than \(n\), the result follows. 
	\end{proof}

	\begin{lemma}\label{lem:a-closed-R[t]}
		Let \(a\in R\) and \(I\subseteq R\) be a left ideal. If \(I\) is \(a\)-closed, then \(I[t] + R[t](1-at)\) is a \(t\)-saturated left ideal of \(R[t]\). 
	\end{lemma}
	\begin{proof}
		
		To prove that \(M = I[t] + R[t](1-at)\) is \(t\)-saturated, let \(f(t)t\in M\) for some \(f(t)\in R[t]\). Then \(f(t)t = g(t) + h(t)(1-at)\), where \(g(t)\in I[t]\) and \(h(t)\in R[t]\). Since \(I\) is \(a\)-closed, we have \(g(t)a\in I[t]\subseteq M\). It follows that
		\[
		f(t)ta = g(t)a + h(t)a(1-at)\in M.
		\]
		Thus \(f(t) = f(t)ta+f(t)(1-at)\in M\), completing the proof. 
	\end{proof}
	
	\begin{proposition}\label{prp:1-1corres}
		For any \(a\in R\), the assignment \(I\mapsto I[t]+ R[t](1-at)\) establishes a one-to-one correspondence between the set of all \(a\)-complete left ideals \(I\) of \(R\) and the set of all \(t\)-saturated left ideals of \(R[t]\) containing \(1-at\). 
	\end{proposition}
	\begin{proof}
		The assignment is well defined by Lemma~\ref{lem:a-closed-R[t]}. 
		The fact that the assignment is one-to-one follows immediately from Proposition~\ref{prp:a-saturated-R[t]}. To finish the proof, we only need to show that if \(M\) is a \(t\)-saturated left ideal of \(R[t]\) containing \(1-at\), then \(M=I[t]+R[t](1-at)\) for some \(a\)-complete left ideal \(I\) of \(R\). Set \(I=M\cap R\). We claim that \(M=I[t]+R[t](1-at)\).  Let \(f(t)\in M\). By Proposition~\ref{prp:division-1-at}, \(f(t)-bt^n\in R[t](1-at)\) for some \(b\in R\). It follows that \(bt^n\in M\), and since \(M\) is \(t\)-saturated, \(b\in M\cap R = I\). Therefore, \(f(t)\in I[t]+R[t](1-at)\), completing the proof of the claim. It only remains to show that \(I\) is \(a\)-complete.
		Since
		\[
		I = M\cap R = (I[t]+R[t](1-at))\cap R, 
		\]
		\(I\) must be \(a\)-saturated by Proposition~\ref{prp:a-saturated-R[t]}. For any \(r\in I\), we have 
		\[
		rat = r-r(1-at) \in M.
		\]
		Since \(M\) is \(t\)-saturated, it follows that \(ra\in M\cap R = I\). This shows that \(Ia\subseteq I\), that is, \(I\) is \(a\)-closed, and we are done. 
	\end{proof}
	
	The following lemma will be needed in later sections. 
	\begin{lemma}\label{lem:elements-of-I[t]+R[t](1-at)}
		Let \(a\in R\) and let \(I\) be an \(a\)-complete left ideal in \(R\). Then, a polynomial \(f(t)=\sum_{i=0}^nr_it^i\in R[t]\) of degree \(n\) belongs to \(I[t] + R[t](1-at)\) iff \(\sum_{i=0}^nr_ia^{n-i}\in I\).
	\end{lemma}
	\begin{proof}
		Set \(c=\sum_{i=0}^n r_i a^{n-i}\) and \(M=I[t]+R[t](1-at)\). Proposition~\ref{prp:division-1-at} gives \(f-ct^n\in M\). Since \(M\) is \(t\)-complete and \(M\cap R=I\), we have \(f\in M\) iff \(\sum_{i=0}^n r_i a^{n-i}=c\in I\).
	\end{proof}
	
	\subsection{Integral dependence over a left ideal}
	
	An element \(a\) of a ring \(R\) is \emph{integral} over a left ideal \(I\subseteq R\) if there exist \(r_0,\dots,r_{n-1}\in I\) such that 
	\[ a^n+r_{n-1}a^{n-1}+\cdots+r_1a+r_0=0. \]
	A subset is \emph{integral} over \(I\) if each of its elements is. By contrast, it is \emph{nil over} \(I\) if every element has a power in \(I\), and is \emph{nilpotent over} \(I\) if all products of some fixed length belong to \(I\).
	A subset \(A\) is \emph{fully integral over} \(I\) if there is \(n\geq1\) such that for every \(a_1,\ldots,a_n\in A\),
	\[
	a_1\cdots a_n\in\sum_{k=0}^{n-1}\ \sum_{1\leq i_1,\ldots,i_k\leq n}Ia_{i_1}\cdots a_{i_k}.
	\]
	The following results, except for Proposition~\ref{prp:sum-nil-ideal}, are recalled from Part I~\cite{aryapoor2026integral}. 	
	\begin{proposition}\label{prp:int-char}
		Let \(R\) be a ring. For \(a\in R\) and a left ideal \(I\subseteq R\), the following statements are equivalent:
		\begin{enumerate}
			\item \(a\) is integral over \(I\).
			\item \(a^n\in Ia^\infty\) for some \(n\ge 1\).
			\item \(t^{n} \in I[t] + R[t](t-a)\) for some \(n\ge 1\).
			\item \(I[t] + R[t](1-at) = R[t]\).
			\item \((I : a^\infty) = R\).
			\item The only \(a\)-saturated left ideal containing \(I\) is  \(R\).
		\end{enumerate}
	\end{proposition} 
	\begin{proposition}\label{prp:sum-nil-ideal}
		Let \(J\) be a nil two-sided ideal of \(R\). For a left ideal \(I\) of \(R\) and \(a\in R\), if \(a\) is integral over \(I+J\), then \(a\) is integral over \(I\). 
	\end{proposition}
	\begin{proof}
		Let \(K=Ia^\infty\). An integral equation over \(I+J\) gives \(a^n=p+z\), with \(p\in K\) and \(z\in J\), since \(J\) is two-sided. Choose \(m\) with \(z^m=0\). Expanding \((a^n-p)^m\), we see that every term other than \(a^{nm}\) belongs to \(K\). Hence \(a^{nm}\in K=Ia^\infty\), and Proposition~\ref{prp:int-char} proves integrality over \(I\). 
	\end{proof}
	\begin{proposition}\label{prp:int-ab-ba}
		Let \(a,b\in R\) and \(I\subseteq R\) be a left ideal. Then \(ab\) is integral over \(I\) iff \(ba\) is integral over the left ideal \(Ia\).
	\end{proposition}
	
	\begin{definition}
		A left ideal \(I\) of a ring \(R\) is said to be \emph{integrally closed (in \(R\))} if no left ideal of \(R\) properly containing \(I\) is integral over \(I\).
	\end{definition}
	\begin{proposition}\label{prp:int-cls-quotient-ring} 
		Let \(J\) be a two-sided ideal of \(R\). A left ideal \(I\supseteq J\) is integrally closed in \(R\) iff \(I/J\) is integrally closed in \(R/J\). 
	\end{proposition}
	
	\begin{definition}
		Let \(I\) be a left ideal of a ring \(R\). The \emph{integral radical} of \(I\), denoted by \(\irad{I}\), is the intersection of all integrally closed left ideals of \(R\) that contain \(I\). The integral radical of the zero ideal, denoted by \(\Int(R)\), is called the  \emph{integral radical} of \(R\). 
	\end{definition}
	
	\begin{proposition}\label{prp:int-in-intrad}
		Let \(I\) and \(J\) be left ideals of \(R\). If \(J\) is integral over
		\(I\), then
		\(
		J \subseteq \irad{I}.
		\)
	\end{proposition}

	\begin{definition}
		A left ideal \(I\) of a ring \(R\) is said to be \emph{fully integrally closed (in \(R\))} if no left ideal of \(R\) properly containing \(I\) is fully integral over \(I\).
	\end{definition}
	
	\begin{proposition}\label{prp:full-implies-semiprime}
		Every fully integrally closed left ideal \(I\) of \(R\) is semiprime, that is, \(aRa\subseteq I\) implies \(a\in I\). The converse holds if \(I\) is two-sided. 
	\end{proposition}
	
	\begin{definition}
		Let \(I\) be a left ideal of a ring \(R\). The \emph{full integral radical} of \(I\), denoted by \(\frad{I}\), is the intersection of all fully integrally closed left ideals of \(R\) that contain \(I\).  
	\end{definition}
	\begin{proposition}\label{prp:full-int-rad-properties}
		\begin{enumerate}
			\item For every left ideal \(I\), \(\frad{I}\)  is the smallest
			fully integrally closed left ideal containing \(I\). 
			\item For every left ideal \(I\), \(\frad{I}\subseteq \irad{I}\).
			\item Let \(I\) and \(J\) be left ideals of \(R\). If \(J\) is fully integral over
			\(I\), then
			\(
			J \subseteq \frad{I}.
			\)
		\end{enumerate}
	\end{proposition}
	\begin{example}\label{exm:matrix-int}
		Let \(D\) be a division ring and \(R=\Mat_n(D)\). By Proposition~\ref{prp:int-char} and Example~\ref{exm:matrix-rings-clo-sat-com}, \(A\) is integral over \(RB\) iff \(W\nsubseteq A(W)\) for every nonzero right \(D\)-subspace \(W\subseteq V_B\). Consequently every left ideal \(RB\) is integrally closed. 
	\end{example}
	The \emph{Jacobson radical} of a left ideal \(I\), denoted by \(\jrad{I}\), is defined as the intersection of all maximal left ideals containing \(I\). The Jacobson radical of the zero ideal of a ring \(R\), denoted by \(\Jac(R)\), is called the \emph{Jacobson radical} of \(R\).
	\begin{proposition}\label{prp:char-jac-rad}
		Let \(I\) be a left ideal of a ring \(R\). An element \(a\in R\) belongs to \(\jrad{I}\) iff \(I+R(1-ra) = R\) for all \(r\in R\).
	\end{proposition}
	\begin{proposition}\label{prp:intrad-in-jac}
		For every left ideal \(I\), we have \(\irad{I}\subseteq \jrad{I}\). In particular, \(\Int(R)\subseteq \Jac(R)\) for any ring \(R\).
	\end{proposition}
	
	\begin{theorem}\label{thm:rad-artinian}
		Let \(R\) be a left artinian ring. For any left ideal \(I\) in \(R\), the left ideal \(\jrad{I}\) is fully integral over \(I\). In particular, \(\frad{I} = \irad{I} = \jrad{I}\).  
	\end{theorem}
	
	\begin{corollary}\label{cor:k-algebraic}
		Let \(R\) be an algebra over a field \(k\). Suppose \(R\) is algebraic over \(k\). Then, for any  left ideal \(I\) in \(R\), \(\jrad{I}\) is the largest left ideal that is integral over \(I\).    In particular, \(\irad{I} = \jrad{I}\). 
	\end{corollary}
	
	\begin{theorem}\label{thm:amitsur-thm}
		Let \(R\) be an algebra over a field \(k\). Suppose \(\dim_k R < |k|\). Then for any left ideal \(I\) in \(R\), \(\jrad{I}\) is the largest left ideal that is integral over \(I\).    In particular, \(\irad{I} = \jrad{I}\). 
	\end{theorem}
	
	\begin{theorem}[Generalized Amitsur Theorem]\label{thm:jac-rad-R[x]}
		For any left ideal \(I\) of \(R\), \(\jrad{I[x]} = (R\cap \jrad{I[x]}\,)[x]\). Furthermore,  \(R\cap \jrad{I[x]}\) is integral over \(I\). 
	\end{theorem}
	
	\subsection{Left ideals of matrix rings}\label{subsec:left-ideals-matrix}
	Fix \(d\geq1\). For a left \(R\)-submodule \(V\subseteq R^d\), define
	\[
	\mathcal L(V)\colonequals\{A\in\Mat_d(R)\mid\text{each row of }A\text{ belongs to }V\}.
	\]
	The assignments \(V\mapsto\mathcal L(V)\) and \(P\mapsto r_1(P)\), where \(r_1(P)\) is the set of first rows, are inverse lattice isomorphisms between the lattice of left \(R\)-submodules of \(R^d\) and the lattice of left ideals of \(\Mat_d(R)\). In what follows, we shall use this isomorphism to relate classes of left ideals of \(\Mat_d(R)\) to classes of left ideals of the polynomial ring \(R[x]\). 

	Fix \(h(x)=\sum_{i=0}^{d-1}a_ix^i\in R[x]\), allowing zero leading coefficients. We define a map \(\phi_h: R[x]\to R[x]\) by 
	\[
	\phi_h(f)=\frac{f-f(0)}x+f(0)h.\]
	In the following proposition, we collect some properties of the map \(\phi_h\). 
	
	\begin{proposition}\label{prp:properties-phif}
		The map \(\phi_h\) is left \(R\)-linear. It satisfies
		\[
		\phi_h(xf)=f,\quad
		f=f(0)(1-xh)+x\phi_h(f),\quad
		\deg\phi_h(f)<\max(\deg f,d).
		\]
		Moreover, for all \(g\in R[x]\),
		\[
		\phi_h(g(1-xh))=\frac{g-g(0)}x(1-xh).
		\]
		Consequently, for \(m\geq1\),
		\[
		f=\left(\sum_{j=0}^{m-1}\phi_h^j(f)(0)x^j\right)(1-xh)
		+x^m\phi_h^m(f).
		\]
	\end{proposition}
	\begin{proof}
		The first four identities follow by substitution. Iterating the second identity gives the last one. 
	\end{proof}
	Given a subset \(S\) of \(R[x]\) and a nonnegative integer \(n\), we shall use the notation \(S_{\deg < n}\) to denote the set
	\[
	S_{\deg < n}\colonequals \{f\in S\mid \deg f < n\}.
	\]
	We identify \(F=R[x]_{\deg<d}\) with row vectors in \(R^d\) via
	\[
	\sum_{i=0}^{d-1}r_ix^i\mapsto \begin{pmatrix}
		r_0&r_1&\cdots&r_{d-1}
	\end{pmatrix}
	\]
	and put
	\[
	C_h=\begin{pmatrix} a_0 & a_1 & a_2 & \cdots & a_{d-2} & a_{d-1} \\ 1 & 0 & 0 & \cdots & 0 & 0 \\ 0 & 1 & 0 & \cdots & 0 & 0 \\ \vdots & \vdots & \ddots & \ddots & \vdots & \vdots \\ 0 & 0 & \cdots & 1 & 0 & 0 \\ 0 & 0 & \cdots & 0 & 1 & 0 \end{pmatrix}\in\Mat_d(R).
	\]
	For \(d=1\), \(C_h=(a_0)\). 
	On \(F\), \(\phi_h\) is right multiplication by \(C_h\). The following proposition establishes a correspondence between a certain class of left ideals of \(R[x]\) and a certain class of submodules of \(F\). 
	\begin{proposition}\label{prp:left-ideal-degree}\label{prp:left-ideal-1-hx}
		There is an inclusion-preserving bijection between the \(x\)-complete left ideals \(M\) of \(R[x]\) containing \(1-xh\) and the left \(R\)-submodules \(V\subseteq F\) satisfying
		\[
		f\in V\quad\Longleftrightarrow\quad\phi_h(f)\in V\qquad(f\in F).
		\]
		Its maps are
		\[
		M\longmapsto M\cap F,\qquad
		V\longmapsto M_V:=\{f:\phi_h^m(f)\in V\text{ for some }m\geq0\}.
		\]
		Equivalently,
		\[
		M_V=R[x](1-xh)+\bigcup_{m\geq0}x^mV.
		\]
	\end{proposition}
	\begin{proof}
		Suppose first that \(V\) has the stated property. Once an iterate belongs to \(V\), all subsequent iterates do. Thus a common sufficiently large iterate proves that \(M_V\) is an additive subgroup stable under left multiplication by \(R\). The identity \(\phi_h(xf)=f\) shows both that \(xM_V\subseteq M_V\) and that \(xf\in M_V\) implies \(f\in M_V\). Hence \(M_V\) is an \(x\)-complete left ideal. The identity \(\phi_h(1-xh)=0\) implies \(1-xh\in M_V\). For \(f\in F\), the equivalence defining \(V\), iterated backwards, gives \(M_V\cap F=V\). The last displayed formula follows from the last identity of Proposition~\ref{prp:properties-phif} and the stability just proved.
		
		Conversely, let \(M\) be \(x\)-complete and contain \(1-xh\). The identity \[f-f(0)(1-xh)=x\phi_h(f)\] gives \(f\in M\) iff \(\phi_h(f)\in M\). This proves the required property of \(V=M\cap F\). Since sufficiently many iterates enter \(F\), it also proves \(M=M_V\). It is esay to see that both assignments preserve inclusions.
	\end{proof}
	Finally, we establish a useful correspondence between left ideals of matrix rings and those of polynomial rings. 
	\begin{proposition}\label{prp:1-1corres-x-complete}
		With the preceding conventions, \(M\mapsto\mathcal{L}(M\cap F)\) is an inclusion-preserving bijection between the \(x\)-complete left ideals of \(R[x]\) containing the polynomial \(1-xh\) and the \(C_h\)-complete left ideals of \(\Mat_d(R)\). Furthermore, \(M\cap R = \mathcal{L}(M\cap F)\cap R \) for every \(x\)-complete left ideal \(M\) containing \(1-xh\).
	\end{proposition}
	\begin{proof}
		Combine the preceding proposition with the row-module correspondence and the identity \(\phi_h(v)=vC_h\). The proof of the second statement is straightforward. 
	\end{proof}


	\section{G-left ideals and G-radicals}\label{sec:G-left}
	
	G-ideals play a central role in the ideal theory of commutative rings and have been the subject of extensive study; see, for example, Kaplansky's classical book \cite{kaplansky2006commutative}. Recall that a G-ideal of a commutative ring is a prime ideal \(P\) such that the intersection of all prime ideals properly containing \(P\) is strictly larger than \(P\). This notion admits a natural generalization to the noncommutative setting, as was done by Resco, Stafford, and Warfield in their proof of a Nullstellensatz for fully bounded G-rings \cite{resco1986fully}. In this section, we propose an alternative generalization of G-ideals that is better suited to one-sided ideals. We also introduce and study the associated G-radicals.
	
	\subsection{G-left ideals}
	Among the various characterizations of G-ideals in the commutative setting, we shall make use of the following one. 
	\begin{proposition} Let \(R\) be a commutative ring. An ideal \(P\) of \(R\) is a G-ideal iff there exists an element \(a\in R\) such that \(a\) is not nilpotent over \(P\), but is nilpotent over every ideal properly containing \(P\). 
	\end{proposition}
	
	Motivated by this characterization of G-ideals, we replace nilpotence over an ideal by integrality over a left ideal and introduce the following generalization.
	
	\begin{definition} Let \(R\) be a ring and \(P\) be a left ideal of \(R\). We say that \(P\) is a \emph{G-left ideal with witness} \(a\in R\) if \(a\) is not integral over \(P\) and is integral over every left ideal \(J\supsetneq P\). A left ideal \(P\) is called a \emph{G-left ideal} if there exists an element \(a\in R\) such that \(P\) is a G-left ideal with witness \(a\). 
	\end{definition}
	
	Any maximal left ideal is a G-left ideal with witness \(1\). In particular, every proper left ideal is contained in at least one G-left ideal. More generally, we have the following result. 
	\begin{proposition}\label{prp:extending-to-G-left}
		Let \(a\in R\) and \(I\subset R\) be a left ideal such that \(a\) is not integral over \(I\). Then, \(I\) is contained in a G-left ideal with witness \(a\).
	\end{proposition}
	\begin{proof}
		Consider the set \(\Sigma\) of all left ideals \(I\subseteq J\subsetneq R\) such that \(a\) is not integral over \(J\). Since \(I\in\Sigma\), the set   \(\Sigma\) is nonempty. Let \(\{I_i\}\subseteq \Sigma\) be a chain of left ideals. It is easy to see that \(a\) cannot be integral over \(\cup_iI_i\), and consequently,  \(\cup_iI_i\in  \Sigma\). By Zorn's lemma, \(\Sigma\) contains a maximal element \(P\) with respect to inclusion. It is clear that \(P\supseteq I\) is  a G-left ideal with witness \(a\), which completes the proof. 
	\end{proof}
	
	The following propositions provide some basic properties of G-left ideals. Recall that a one-sided ideal \(I\) is prime if \(aRb\subseteq I\) implies \(a\in I\) or \(b\in I\).
	\begin{proposition}\label{prp:properties-G-left}
		Let \(P\) be a G-left ideal with witness \(a\). Then \(P\) is \(a\)-complete and  prime.
	\end{proposition}
	\begin{proof}
		If \(a\) were integral over \(Pa^\infty\), it would be integral over \(P\), by Proposition~\ref{prp:int-char}. Hence maximality gives \(Pa^\infty=P\). The proper \(a\)-saturation \((P:a^\infty)\) likewise equals \(P\). This proves completeness.
		
		To prove the second statement, suppose \(bRc\subset P\) and \(c\notin P\). Then \(P\nsubseteq P + Rc\), implying that \(a\) is integral over \(P + Rc\). Therefore \(a^n\in (P+Rc)a^\infty\) for some \(n\geq 1\). Then
		\[
		ba^n \in b(P+Rc)a^\infty\subseteq (P+bRc)a^\infty\subseteq P.
		\] 
		Since \((P:a) = P\), we conclude that \(b\in P\), which completes the proof. 
	\end{proof}
	
	The proofs of the following propositions are straightforward and are therefore omitted.
	\begin{proposition}\label{prp:P-phi(P)}
		Let \(P\subset R\) be a G-left ideal with witness \(a\). For every automorphism \(\phi:R\to R\), the left ideal \(\phi(P)\) is a G-left ideal with witness \(\phi(a)\). In particular, for every unit \(u\in R\), the left ideal \(u^{-1}Pu=Pu\) is a G-left ideal with witness \(u^{-1}au\).
	\end{proposition}
	
	\begin{proposition}\label{prp:G-quotient-ring}
		Let \(J\) be a two-sided ideal in \(R\). For any left ideal \(P\supseteq J\) and \(a\in R\), \(P\) is a G-left ideal with witness \(a\) iff \(P/J\) is a G-left ideal of \(R/J\) with witness \(a+J\). 
	\end{proposition} 
	\subsection{Characterizations of G-left ideals}
	In the following proposition, we give several characterizations of G-left ideals. 
	\begin{proposition}\label{prp:G-left-characterization}
		Let \(a\in R\) and let \(P\subset R\) be a proper left ideal. Then the following statements are equivalent:
		\begin{enumerate}
			\item \(P\) is a G-left ideal with witness \(a\).
			\item \(P[t] + R[t](1-at) \neq R[t]\), and for every left ideal \(I\) properly containing \(P\), we have
			\[
			I[t] + R[t](1-at) = R[t].
			\]
			\item \((P:a^\infty)\neq R\) and \((I:a^\infty) = R\) for every left ideal \(I\) properly containing \(P\).
			\item \((P:a^\infty) = P\) and \(R\) is the only \(a\)-saturated left ideal properly containing \(P\).
		\end{enumerate}
	\end{proposition}
	
	\begin{proof}
		The proposition follows immediately from Proposition~\ref{prp:int-char}.
	\end{proof}
	
	The following proposition provides another characterization of G-left ideals and may be viewed as a one-sided analogue of the mechanism underlying the Rabinowitsch trick.
	\begin{proposition}\label{prp:char-G-left-R[t]}
		A left ideal \(P\subseteq R\) is a G-left ideal with witness \(a\in R\) iff \(P\) is \(a\)-complete and \(P[t] + R[t](1-at)\) is  a maximal left ideal in \(R[t]\). 
	\end{proposition}
	\begin{proof}
		Suppose \(P\) is G-left with witness \(a\). Note that \(P\) is \(a\)-complete by Proposition~\ref{prp:properties-G-left}. Put \(M=P[t]+R[t](1-at)\). Its properness follows from 
		Proposition~\ref{prp:G-left-characterization},
		and its \(t\)-completeness follows from Proposition~\ref{prp:1-1corres}.  For \(f\notin M\) of degree \(n\), put \(c=\sum r_i a^{n-i}\) if \(f=\sum r_it^i\). Then \(f-ct^n\in M\) by Proposition~\ref{prp:division-1-at} and \(c\notin P\) by Lemma~\ref{lem:elements-of-I[t]+R[t](1-at)}. Since \(a\) is integral over \(P+Rc\) and \((P:a)=P\), we can write
		\[
		a^m=p+\sum_{j=0}^{m}s_jca^j
		\]
		for some \(p\in P, s_j\in R\).
		Set 
		\[
		g(t) = \sum_{j=0}^{m}s_{m-j}t^j
		\]
		and \( h(t)  = 1-g(t)f(t) \). We have
		\[
		a^{m+n} - \sum_{i=0}^{m} a^ns_{i}ca^i = a^np\in P.
		\]
		By Lemma \ref{lem:elements-of-I[t]+R[t](1-at)}, \(1-g(t)f(t)=h(t)\in P[t] + R[t](1-at)\). This proves 
		\[
		P[t] + R[t](1-at) + R[t] f = R[t].
		\]
		Since \(f(t)\in R[t]\setminus P[t] + R[t](1-at)\) was arbitrary, we conclude that   the left ideal \(P[t] + R[t](1-at)\) is  a maximal left ideal in \(R[t]\).
		
		Conversely, if \(P\) is \(a\)-complete and \(M=P[t] + R[t](1-at)\) is maximal, then \(M\cap R=P\) by Proposition~\ref{prp:a-saturated-R[t]}. Thus \(a\) is not integral over \(P\) by Proposition~\ref{prp:int-char}. For \(J\supsetneq P\), the ideal \(J[t]+R[t](1-at)\) properly contains \(M\), and consequently, it must be the whole ring. Therefore \(a\) is integral over \(J\) by Proposition~\ref{prp:int-char}. This proves that \(P\) is a G-left ideal, which completes the proof. 
	\end{proof}
	Using this proposition, we obtain the following result. 
	\begin{proposition}\label{prp:1-1corres-G-left}
		Let \(a\in R\). Then the assignment \(P\mapsto P[t]+ R[t](1-at)\) establishes a one-to-one correspondence between the set of all G-left ideals \(P\subset R\) with witness \(a\)   and the set of all maximal left ideals of \(R[t]\) containing \(1-at\).  
	\end{proposition}
	\begin{proof}
		This is an immediate consequence of  Propositions~\ref{prp:1-1corres} and \ref{prp:char-G-left-R[t]}. Note that every maximal left ideal containing \(1-at\) is \(t\)-saturated.   
	\end{proof}

	The following example gives a description of G-left ideals in matrix rings over division rings.
	\begin{example}\label{exm:matrix-G-left}
		Let \(D\) be a division ring and set \(R=\Mat_n(D)\). 
		Using Part 4 of Proposition~\ref{prp:G-left-characterization} and the results from Example \ref{exm:matrix-rings-clo-sat-com}, the reader can verify that a proper left ideal \(RB\) is a G-left ideal with witness \(A\) iff \(A(V_B) = V_B\) and \(A\) has no nontrivial invariant \(D\)-subspaces in \(V_B\). In the case where \(D=k\) is a field, we have the following result: a proper left ideal \(RB\) is a G-left ideal iff there exists an irreducible polynomial \(f(x)\in k[x]\) of degree \(\dim_k V_B\). 
	\end{example}

	\subsection{The G-radical of a left ideal}
	The concept of a G-left ideal can be used to define a radical for left ideals. More precisely, we introduce the following concept.
	\begin{definition}
		The \emph{G-radical} of a left ideal \(I\) in a ring \(R\), denoted by \(\grad{I}\), is the intersection of all G-left ideals of \(R\) containing \(I\). 
	\end{definition}
	Regarding the interplay between the various radicals of left ideals, we have the following result.
	\begin{proposition}\label{prp:incl-grad-irad-jrad}
		For any left ideal \(I\) of \(R\), we have
		\[
		\grad{I}\subseteq \irad{I}\subseteq \jrad{I}.
		\]
		In particular, \(\grad{(0)}\subseteq \Int(R)\subseteq \Jac(R)\).
	\end{proposition}
	\begin{proof}
		
		To prove \(\grad{I}\subseteq \irad{I}\), let \(a\notin \irad{I}\). We need to show that \(a\notin \grad{I}\). we can choose an integrally closed left ideal \(J\supseteq I\) such that \(a\notin J\). Since \(J\) is integrally closed and \(a\notin J\), \(Ra\) is not integral over \(J\). It follows that there exists \(r\in R\) such that \(ra\) is not integral over \(J\). Consequently, \(ra\) is not integral over \(I\). By Propositions~\ref{prp:extending-to-G-left}, there exists a G-left ideal \(P\) such that \(ra\notin P\supseteq I\). It follows that \(ra\notin \grad{I}\), and consequently, \(a\notin \grad{I}\).

		The inclusion \(\irad{I}\subseteq \jrad{I}\) holds by Propositions~\ref{prp:intrad-in-jac}. For the second statement, apply the first statement to \(I=(0)\).
	\end{proof}
	G-radicals enjoy the following properties. Recall that a one-sided ideal \(I\) is semiprime if \(aRa\subseteq I\) implies \(a\in I\).
	\begin{proposition}\label{prp:grad-semi-int}
		For every left ideal \(I\) of \(R\), the following statements hold: 
		\begin{enumerate}
			\item The left ideal \(\grad{I}\) is  semiprime . 
			\item The left ideal \(\grad{I}\) is integral over \(I\).
		\end{enumerate}
		
	\end{proposition}
	\begin{proof}
		(1) By Proposition~\ref{prp:properties-G-left}, every G-left ideal is  prime, and intersections of  semiprime  left ideals are  semiprime.
		
		(2) It suffices to show that if \(a\in R\) is not integral over \(I\), then \(a\notin \grad{I}\). To prove this, suppose that \(a\in R\) is not integral over \(I\). Applying Proposition~\ref{prp:extending-to-G-left}, we can enlarge \(I\) to a left ideal \(P\) such that \(P\) is a G-left ideal with witness \(a\). Since \(\grad{I}\subseteq P\) and  \(a\) is not integral over \(P\), \(a\) cannot belong to \(\grad{I}\). This completes the proof. 
	\end{proof}
	As an immediate consequence of the proposition, we have the following result
	\begin{proposition}\label{prp:G-rad-int-closed}
		For every integrally closed left ideal \(I\), \(\grad{I} = I\). In particular, \(I\) 
		is an intersection of G-left ideals.
	\end{proposition}

	Next we show that the G-left radical of the zero ideal coincides with the classical upper nilradical \(\Nil^*(R)\), which is defined as the largest nil two-sided ideal in \(R\). First, we need some lemmas.

	\begin{lemma}\label{lem:G-radical-automorphism}
		For any automorphism \(\phi:R\to R\), we have 
		\[
		\phi(\grad{(0)})=\grad{(0)}.
		\]
		In particular, for any unit \(u\in R\), \(\grad{(0)}\,u = \grad{(0)}\).
	\end{lemma}
	\begin{proof}
		The  statement follows easily from Proposition~\ref{prp:P-phi(P)}. 
	\end{proof}
	
	\begin{lemma}\label{lem:units}
		Let $I$ be a left ideal of $R$ satisfying the following three conditions:
		\begin{enumerate}
			\item $I\subseteq\Jac(R)$;
			\item $Iu\subseteq I$ for every unit $u\in R$;
			\item $I$ is semiprime.
		\end{enumerate}
		Then $I$ is a two-sided ideal.
	\end{lemma}
	
	\begin{proof}
		Let  $a\in I$ and $r\in R$. We must prove that $ar\in I$. For arbitrary
		$s\in R$, put \(b=rsa.\) 
		Since \(br\in \Jac(R)\), \(1-br\) is a unit, and consequently,  $a(1-br)\in I$. We have
		\[
		(ar)s(ar)=arsar=abr=a-a(1-br)\in I.
		\]
		This holds for every $s\in R$, so $(ar)R(ar)\subseteq I$. Since \(I\) is semiprime, we conclude that $ar\in I$, which completes the proof. 
	\end{proof}	
	
	\begin{theorem}\label{thm:G-radical.2-sided}
		For any ring \(R\), the intersection of all the G-left (resp., G-right) ideals of \(R\) is equal to the upper nilradical \(\Nil^*(R)\).
	\end{theorem} 
	\begin{proof}
		By Proposition~\ref{prp:incl-grad-irad-jrad}, Lemma~\ref{lem:G-radical-automorphism} and Proposition~\ref{prp:grad-semi-int},  the left ideal \(\grad{(0)}\) satisfies the properties listed in Lemma~\ref{lem:units}. Consequently, \(\grad{(0)}\) is a two-sided ideal. Proposition~\ref{prp:grad-semi-int} makes it nil, so \(\grad{(0)}\subseteq\Nil^*(R)\). To prove the reverse inclusion, let \(P\) be a G-left ideal with witness \(a\). If \(\Nil^*(R)\nsubseteq P\), then \(a\) is integral over \(P+\Nil^*(R)\). Proposition~\ref{prp:sum-nil-ideal} makes \(a\) integral over \(P\), a contradiction. Hence every G-left ideal contains \(\Nil^*(R)\), this proves \(\Nil^*(R)\subseteq\grad{(0)}\), hence \(\grad{(0)}=\Nil^*(R)\). Applying the left-handed result to the opposite ring proves the right-handed statement.
	\end{proof}
	
	\begin{remark}
		The theorem generalizes the following well-known result from the commutative setting to the noncommutative setting: the nilradical of a commutative ring is the intersection of all the G-ideals of the ring \cite[Theorem 25]{kaplansky2006commutative}.  
	\end{remark}
	
	We record the following proposition regarding the behavior of G-radicals under quotient rings. 
	\begin{proposition}
		Let \(J\) be a two-sided ideal of \(R\). Then, for every left ideal
		\(I\) of \(R\) containing \(J\),
		\[
		\frac{\Nil^*(R)+I}{J}\subseteq \grad{\,I/J\,}=\frac{\grad{I}}{J}.
		\]
	\end{proposition}
	\begin{proof}
		The equality follows directly from Proposition~\ref{prp:G-quotient-ring}. The inclusion is a consequence of  the fact that \(\grad{I}\) contains both \(I\) and \(\grad{(0)}=\Nil^*(R)\). 
	\end{proof}
	
	\subsection{Characterizations of the G-radical}
	To present the first characterization, we need to introduce a concept. Let \(I\) be a left ideal in \(R\). A left ideal \(J\) of \(R\) is called \emph{\(I\)-negligible} if the following condition holds: for every left ideal \(K\) containing \(I\) and \(a\in R\), if \(a\) is integral over \(K+J\), then \(a\) is already integral over \(K\); or equivalently, if \(a\) is not integral over \(K\), then \(a\) is not integral over \(K+J\). In the case where \(I=(0)\), we use the term ``negligibl'' in place of \((0)\)-negligible.
	
	It is easy to see that if \(I\subseteq I'\), then any \(I\)-negligible left ideal is also \(I'\)-negligible. In particular, any negligible left ideal is \(I\)-negligible for every left ideal \(I\). 
	\begin{proposition}
		For every left ideal \(I\), the sum of any family of \(I\)-negligible left ideals is \(I\)-negligible. 
	\end{proposition}
	\begin{proof}
		A simple induction proves the result for finite families. The general case follows from the finite case and the observation that if \(a\) is integral over a sum \(\sum_i I_i\) of left ideals, then \(a\) is integral over the sum of a finite family of left ideals in \(\{I_i\}\).   
	\end{proof}
	The following gives a characterization of G-radicals in terms of the concept of negligibility. 
	\begin{proposition}\label{prp:negligible}
		For any left ideal \(I\) in \(R\), \(\grad{I}\) is the largest \(I\)-negligible left ideal in \(R\). In particular, \(\grad{(0)}=\Nil^*(R)\) is the largest negligible left (right) ideal in \(R\). 
	\end{proposition}
	
	\begin{proof}
		First we show that  \(\grad{I}\) is \(I\)-negligible. Assume that \(a\) is not integral over a left ideal \(K\) containing \(I\). Then \(K\) can be enlarged to a G-left ideal \(P\) with witness \(a\) by Proposition~\ref{prp:extending-to-G-left}. Then \(a\) is not integral over \(K + \grad{I}\subseteq P\). This proves that \(\grad{I}\) is \(I\)-negligible.
		
		To show that \(\grad{I}\) is the largest \(I\)-negligible left ideal in \(R\), it suffices to show that if \(J\nsubseteq \grad{I}\) is a left ideal, then \(J\) cannot be \(I\)-negligible. In fact, \(J\nsubseteq \grad{I}\) implies that \(J\) is not contained in a G-left ideal \(P\) containing \(I\). Let \(P\) be a G-left ideal with witness \(a\in R\).  Then, \(a\) is integral over \(P+J\), but not over \(P\). So \(J\) is not \(I\)-negligible. 
	\end{proof}
	
	The second characterization is in terms of the polynomial ring \(R[t]\). 
	\begin{proposition}\label{prp:char-G(R)}
		Let \(I\) be a left ideal in \(R\). For  \(b\in R\), the following statements are equivalent:
		\begin{enumerate}
			\item \(b\in \grad{I}\). 
			\item \(I[t]+ R[t](1-at)+R[t](1-fb)  = R[t]\) for all \(f\in R[t]\) and \(a\in R\).
		\end{enumerate}
	\end{proposition}
	\begin{proof}
		(1)\(\Rightarrow\)(2): 
		Let \(b\in \grad{I}\), but
		\[
		J=I[t]+ R[t](1-at)+R[t](1-fb)   \neq  R[t]. 
		\]
		for some \(f\in R[t]\) and \(a\in R\). Then there exists a maximal left ideal \(M\subset R[t]\) that contains the proper left ideal  \(J\) of \(R[t]\). By Proposition~\ref{prp:1-1corres-G-left}, \(M\cap R\) is a G-left ideal in \(R\) that contains \(I\), and therefore, \(b\in M\cap R\). Since \(1-fb\in M\), we have \(1=1-fb+fb\in M\), a contradiction.
		
		(2)\(\Rightarrow\)(1):  Let \(I\subseteq P\subset R\) be a G-left ideal with witness \(a\in R\). By Proposition~\ref{prp:1-1corres-G-left}, the left ideal
		\[
		M = P[t] + R[t] (1-at)
		\]
		is maximal in \(R[t]\) and satisfies \(M\cap R = P\). If \(b\notin P\), then \(1-fb\in M\) for some \(f\in R[t]\), yielding the contradiction
		\[
		R[t] = I[t]+ R[t](1-at)+R[t](1-fb)  \subseteq M.
		\]
		So \(b\in P\), which implies that \(b\in \grad{I}\).  
	\end{proof}
	Using the description of elements of the Jacobson radical of a left ideal, given in Proposition~\ref{prp:char-jac-rad}, we can reformulate the second characterization as follows.
	\begin{corollary}
		For any left ideal \(I\) in \(R\), we have
		\[
		\grad{I} = R\cap \bigcap_{a\in R} \jrad{I[t]+R[t](1-at)}.
		\]
		In particular, 
		\[
		\Nil^*(R) = \grad{(0)} = R\cap \bigcap_{a\in R} \jrad{R[t](1-at)}.
		\]
	\end{corollary}

	\subsection{G-left ideals of matrix rings}
	This subsection completes the passage from polynomial ideals to matrix ideals initiated in Subsection~\ref{subsec:left-ideals-matrix}. The companion matrix records the inverse of multiplication by the polynomial variable on the relevant finite set of coefficients. A second argument, using finite centralizing extensions, gives the converse contraction statement needed in Subsection~\ref{subsec:left-ideals-matrix}.
	We begin with a lemma. 
	\begin{lemma}\label{lem:max-complete}
		A proper \(a\)-complete left ideal is G-left with witness \(a\) iff it is maximal among proper \(a\)-complete left ideals.
	\end{lemma}
	\begin{proof}
		This follows immediately from from Proposition~\ref{prp:char-G-left-R[t]} and Proposition~\ref{prp:1-1corres-G-left}.
	\end{proof}
	The following proposition provides a description of G-left ideals of matrix rings in terms of maximal left ideals of polynomial rings. The notation is the same as that used in Section~\ref{sec:back-pre}.
	
	\begin{proposition}\label{prp:1-1corres-G-left-matrix}
		Let \(h\in R[x]\) have degree less than \(d\), with \(d\geq1\), and let \(C_h\in\Mat_d(R)\) be as in Subsection~\ref{subsec:left-ideals-matrix}. Then
		\[
		M\longmapsto\mathcal L(M\cap R[x]_{\deg<d})
		\]
		is a bijection from the maximal left ideals of \(R[x]\) containing \(1-xh\) to the G-left ideals of \(\Mat_d(R)\) with witness \(C_h\).
	\end{proposition}
	\begin{proof}
		Every maximal left ideal \(M\) containing \(1-xh\) is \(x\)-complete. Indeed, multiplication by central \(x\) on the simple module \(R[x]/M\) is nonzero, since \(x\in M\) would imply \(1\in M\), and hence is invertible. Conversely, any maximal proper \(x\)-complete left ideal containing \(1-xh\) is maximal as a left ideal: enlarge it to a maximal left ideal, which is again \(x\)-complete by the same argument. The result now follows from applying the lattice isomorphism of Proposition~\ref{prp:1-1corres-x-complete} and Lemma~\ref{lem:max-complete}.
	\end{proof}
	Next we study contractions to \(R\) of maximal left ideals of the polynomial ring \(R[x]\) and contractions to \(R\) of G-left ideals of \(\Mat_d(R)\).
	\begin{proposition}\label{prp:max-to-matrix}
		For every maximal left ideal \(M\) of \(R[x]\), there are \(d\geq1\) and a G-left ideal \(P\) of \(\Mat_d(R)\) such that \(M\cap R=P\cap R\).
	\end{proposition}
	\begin{proof}
		If \(x\in M\), then \(M\cap R\) is maximal in \(R\), so take \(d=1\) and \(P=M\cap R\). Otherwise maximality gives \(1-xh\in M\) for some \(h\in R[x]\). Choose \(d>\deg h\) and use the preceding proposition to obtain a G-left ideal 
		\[
		P = \mathcal L(M\cap R[x]_{\deg<d})
		\]
		of \(\Mat_d(R)\). 
		It is easy to see that \(M\cap R = P\cap R\), which completes the proof.
	\end{proof}
	For the next result, we need the following lemma. Its proof can be found in Passman's book \cite[Lemma 20.1]{passman2004course}.
	\begin{lemma}\label{lem:finite-centralizing}
		Let \(R\) be a subring of a ring \(S\) (sharing the same identity) and let  \(S=\sum_{i=1}^q R s_i\), where every \(s_i\) commutes with \(R\). Any simple left \(S\)-module is semisimple of finite length as an \(R\)-module.
	\end{lemma}
	
	\begin{proposition}\label{prp:finite-extension-contraction}
		Let \(S=\sum_{i=1}^q Rs_i\) be a ring extension in which each \(s_i\) commutes with \(R\). For every G-left ideal \(P\) of \(S\), there are finitely many maximal left ideals \(M_1,\ldots,M_m\) of \(R[x]\) such that
		\[
		P\cap R=\bigcap_{i=1}^m(M_i\cap R).
		\]
	\end{proposition}
	\begin{proof}
		Choose a witness \(a\in S\) and put \(N=P[x]+S[x](1-ax)\). The \(S[x]\)-module \(W=S[x]/N\) is simple, and \(P=\operatorname{ann}_S(v)\) for \(v=1+N\). Applying Lemma~\ref{lem:finite-centralizing} to \(R[x]\subseteq S[x]\) shows that \(W\) is a semisimple \(R[x]\)-module. Write \(W\) as a finite direct sum of simple \(R[x]\)-modules and write \(v=\sum v_i\) accordingly. For each nonzero component \(v_i\),
		\[
		M_i=\operatorname{ann}_{R[x]}(v_i)
		\]
		is a maximal left ideal of \(R[x]\). Taking annihilators of the direct-sum components proves the formula.
	\end{proof}
	We record the following corollary. 
	\begin{corollary}\label{cor:matrix-to-max}
		For every G-left ideal \(P\) of \(\Mat_d(R)\), its contraction to \(R\) is a finite intersection of contractions of maximal left ideals of \(R[x]\).
	\end{corollary}
	\begin{proof}
		Apply Proposition~\ref{prp:finite-extension-contraction} to the ring extension \(R\subseteq M_d(R)\).
	\end{proof}
	
	\section{KI-rings}\label{sec:kothe}
	The classical K\"othe problem concerns the behavior of nil one-sided ideals under addition. Integral dependence allows the zero ideal to be replaced by an arbitrary left ideal, giving rise to the concept of a left KI-ring. In this section, we study KI-rings, and in particular,  give several characterizations of KI-rings.
	
	\subsection{KI-rings and their characterizations}
	For a fixed ring \(R\), the K\"othe property says that every nil left ideal is contained in the upper nilradical \(\Nil^*(R)\). Equivalently, sums of two nil left ideals are nil. Generalizing this property to integral dependence over arbitrary left ideals, we introduce the following concept. 
	
	\begin{definition}
		We call a ring \(R\) a \emph{left K\"othe ring for integral independence} (or \emph{left KI-ring} for short) if, for each left ideal \(I\subseteq R\), the sum of two left ideals integral over \(I\) is integral over \(I\). The notion of a \emph{right KI-ring} is defined analogously. 
	\end{definition}
	
	The following proposition provides several characterizations of KI-rings. 
	
	\begin{proposition}\label{prp:gen-kothe-1st-equiv}
		For a fixed ring \(R\), the following statements are equivalent.
		\begin{enumerate}[label=(\roman*)]
			\item For every left ideal \(I\) of \(R\), \(\irad{I}\) is integral over \(I\).
			\item \(R\) is a left KI-ring.
			\item Every G-left ideal of \(R\) is integrally closed in \(R\).
			\item If \(J\subseteq R\) is a left ideal integral over a left ideal \(I\subseteq R\), and \(a\) is integral over \(J\), then \(a\) is integral over \(I\).
			\item \(\grad{I}=\irad{I}\) for every left ideal \(I\) of \(R\).
		\end{enumerate}
	\end{proposition}
	\begin{proof}
		The implication (i)\(\Rightarrow\)(ii) follows from Proposition~\ref{prp:int-in-intrad}. Under (ii), finite sums of integral left ideals are integral; arbitrary sums are also integral because each of their elements belongs to a finite subsum.
		
		To prove (ii)\(\Rightarrow\)(iii), let \(P\) have witness \(a\), and suppose \(Rb\) is integral over \(P\), with \(b\notin P\). Then \(a\) is integral over \(P+Rb\), so
		\[
		a^n\in (P+Rb)a^\infty
		\]
		for some \(n\). For every \(r\in R\) and \(j\geq0\), the element \(a^jrb\) is integral over \(P\). Proposition~\ref{prp:int-ab-ba} implies that \(rba^j\) is integral over \(Pa^j\subseteq P\). Thus each \(Rba^j\), and consequently their sum with \(P\), is integral over \(P\). In particular \(a^n\) is integral over \(P\). Substituting \(a^n\) into its integral equation makes \(a\) integral over \(P\), a contradiction.
		
		For (iii)\(\Rightarrow\)(iv), if \(a\) is not integral over \(I\), we can enlarge \(I\) to a G-left ideal \(P\) with witness \(a\). Since \(P\) is integrally closed, \(J\subseteq P\), contradicting integrality of \(a\) over \(J\). For (iv)\(\Rightarrow\)(iii), any proper integral extension of a G-left ideal would, by transitivity, make its witness integral over it.
		
		Under (iii), every G-left ideal containing \(I\) contains \(\irad{I}\), so \(\irad{I}\subseteq\grad{I}\). Proposition~\ref{prp:incl-grad-irad-jrad} gives  the opposite inclusion. Thus (iii) implies (v), and (v) implies (i) by Proposition~\ref{prp:grad-semi-int}. 
	\end{proof}
	
	The following result gives a more complete description of G-radicals of left ideals in KI-rings.  
	\begin{corollary}\label{cor:radical-KI-ring}
		If \(R\) is a left KI-ring, then for every left ideal \(I\),
		\[
		\grad{I}=\irad{I}
		=\sum\{J:J\text{ is a left ideal integral over }I\}.
		\]
		This is the largest left ideal integral over \(I\). 
	\end{corollary}
	\begin{proof}
		The sum is integral over \(I\) by the finite-support argument, and every summand is contained in \(\irad{I}\). Conversely, \(\irad{I}\) is itself integral over \(I\). The equality with \(\grad{I}\) follows from Proposition~\ref{prp:gen-kothe-1st-equiv}.
	\end{proof}
	
	The well-known equivalences of the classical K\"othe conjecture with nil matrix rings and Jacobson radicals of polynomial rings  are normally stated with a universal quantifier over rings; see Krempa~\cite{krempa1972logical} and Smoktunowicz~\cite{agata2001some}. Generalizing these equivalences, we work with fixed-ring formulations,  keeping track of the rings to which each hypothesis is applied. 
	\begin{proposition}\label{prp:gen-kothe-3rd-equiv}
		For a fixed ring \(R\), the following properties are equivalent.
		\begin{enumerate}[label=(\roman*)]
			\item For every \(d\geq1\), the contraction to \(R\) of every G-left ideal of \(\Mat_d(R)\) is integrally closed in \(R\).
			\item The contraction to \(R\) of every maximal left ideal of \(R[x]\) is integrally closed in \(R\).
			\item For every left ideal \(I\subseteq R\), \(\jrad{I[x]}=\irad{I}[x]\).
			\item If \(J\) is a left ideal integral over \(I\subseteq R\), then \(J[x]\subseteq\jrad{I[x]}\).
			\item If \(J\) is a left ideal integral over a left ideal \(I\subseteq R\), then \(\Mat_d(J)\) is integral over \(\Mat_d(I)\) for every \(d\geq1\).
		\end{enumerate}
		Each of these properties implies that \(R\) is a left KI-ring.
	\end{proposition}
	\begin{proof}
		We first prove the equivalence (i)\(\Leftrightarrow\)(ii). Proposition~\ref{prp:max-to-matrix} proves (i)\(\Rightarrow\)(ii). Corollary~\ref{cor:matrix-to-max}, together with the stability of integral closedness under intersections, proves the converse.
		
		We next prove the equivalence (ii)\(\Leftrightarrow\)(iii). Under (ii), the left ideal \(K=R\cap\jrad{I[x]}\) is integrally closed, because it is an intersection of such contractions. Hence \(\irad{I}\subseteq K\). The generalized Amitsur theorem (Theorem~\ref{thm:jac-rad-R[x]}) gives the reverse inclusion \(K\subseteq\irad{I}\). Its polynomial conclusion proves (iii). Conversely, apply (iii) with \(I=M\cap R\), where \(M\) is a maximal left ideal of \(R[x]\): since \(I[x]\subseteq M\), we have
		\[
		\irad{I}[x] = \jrad{I[x]} \subseteq M,
		\]
		which implies \(\irad{I}\subseteq M\cap R=I\), proving (ii).
		
		We next prove the equivalence (iii)\(\Leftrightarrow\)(iv). The implication (iii)\(\Rightarrow\)(iv) follows from \(J\subseteq\irad{I}\) (see Proposition~\ref{prp:int-in-intrad}). Under (iv), sums of two integral left ideals lie in \(R\cap\jrad{I[x]}\), which is integral over \(I\) by the generalized Amitsur theorem (Theorem~\ref{thm:jac-rad-R[x]}); thus \(R\) is a left KI-ring. Proposition~\ref{prp:gen-kothe-1st-equiv} now makes \(\irad{I}\) integral over \(I\). Apply (iv) to \(J=\irad{I}\) and use the generalized Amitsur theorem (Theorem~\ref{thm:jac-rad-R[x]}) for the reverse inclusion, proving (iii).
		
		To prove (i)\(\Rightarrow\)(v), assume (i) and suppose  \(J\) is integral over a left ideal \(I\) but \(A\in\Mat_d(J)\) is not integral over \(\Mat_d(I)\). Choose a G-left ideal \(P\supseteq\Mat_d(I)\) with witness \(A\). Its contraction \(Q=P\cap R\) is integrally closed and contains \(I\), whence \(J\subseteq Q\). But 
		\[\Mat_d(J)\subseteq \Mat_d(Q)\subseteq P,\]
		contradicting \(A\notin P\). This proves (v).
		
		It remains to prove (v)\(\Rightarrow\)(iv). Suppose \(J\) is integral over \(I\). We first show
		\begin{equation}\label{eq:constant-one}
			I[x]+R[x]f(x)=R[x],\qquad
			f(x)=1+r_1x+\cdots+r_dx^d,\quad r_i\in J.
		\end{equation}
		Put \(g(x)=x^d+r_1x^{d-1}+\cdots+r_d\). On the free left \(R\)-module \(R[x]/R[x]g\), use the basis \(1,x,\ldots,x^{d-1}\) and row coordinates. Right multiplication by \(x\) has a matrix \(A\). Reduction by the monic polynomial \(g\) shows that \(A^d\in\Mat_d(J)\): every coefficient in the reductions of \(x^d,\ldots,x^{2d-1}\) is a sum of words in the \(r_i\) of positive length. Hence (v) makes \(A^d\), and therefore \(A\), integral over \(\Mat_d(I)\). Applying an integral matrix equation to the row representing \(1\) gives
		\[
		x^N+u(x)=v(x)g(x),\qquad u\in I[x],\quad v\in R[x].
		\]
		Here the degree of \(u\) need not be less than \(N\). In the Laurent polynomial ring, substitute \(x^{-1}\) and multiply by \(x^N\); clearing negative powers gives
		\[
		x^k\in I[x]+R[x]f(x)
		\]
		for some \(k\geq0\). Since \(f(0)=1\), there is a polynomial \(w\) with \(wf\equiv1\pmod{x^kR[x]}\), obtained by truncating the formal geometric series for \(f^{-1}\). This proves \eqref{eq:constant-one}.
		
		Finally let \(p\in J[x]\) and \(q\in R[x]\). The nonconstant coefficients of \(1-qp\) lie in \(J\), and its constant coefficient is \(1-b\) for some \(b\in J\). Since \(J\subseteq\jrad{I}\), we can choose \(c\in I\) and \(r\in R\) with \(c+r(1-b)=1\). The polynomial 
		\[c+r(1-qp)\in I[x]+R[x](1-qp)\] 
		has constant term \(1\) and all other coefficients in \(J\). Equation~\eqref{eq:constant-one} implies 
		\[I[x]+R[x](1-qp)=R[x].\]
		The Jacobson criterion of Proposition~\ref{prp:char-jac-rad} proves \(p\in\jrad{I[x]}\), which completes the proof. 
	\end{proof}  
	\begin{proposition}\label{prp:gen-kothe-2nd-equiv}\label{thm:gen-kothe-equiv}
		For any ring \(R\), the following statements are equivalent.
		\begin{enumerate}[label=(\roman*)]
			\item Every \(\Mat_k(R)\), \(k\geq1\), is a left KI-ring.
			\item For every \(k\geq1\), the equivalent properties of Proposition~\ref{prp:gen-kothe-3rd-equiv} hold with base ring \(S=\Mat_k(R)\).
		\end{enumerate}
	\end{proposition}
	\begin{proof}
		Statement (ii) implies (i) by Proposition~\ref{prp:gen-kothe-3rd-equiv}. Conversely, suppose (i) and let \(S=\Mat_k(R)\), where \(k\geq 1\). We only need to prove that if \(J\) is a left ideal of \(S\) integral over a left ideal \(I\subseteq S\), then \(\Mat_d(J)\) is integral over \(\Mat_d(I)\) for every \(d\geq1\). Decompose \(\Mat_d(J)\) as the sum of its \(d\) column left ideals. Since \(\Mat_d(S)\cong\Mat_{dk}(R)\) is a KI-ring, it suffices to prove integrality of each column ideal over \(\Mat_d(I)\).
		
		For the first column, let 
		\[
		A = \begin{pmatrix} a_1 & 0 & \cdots & 0 \\ a_2 & 0 & \cdots & 0 \\ \vdots & \vdots & \ddots & \vdots \\ a_d & 0 & \cdots & 0 \end{pmatrix}\in \Mat_d(J).
		\]
		Since \(a_1\) is integral over \(I\), there exist \(r_0,\dots,r_{m-1}\in I\) such that 
		\[ a_1^m+r_{m-1}a_1^{m-1}+\cdots+r_1a_1+r_0=0.\]  
		Note that \(Ara_1^{m}=ArA^m\) for all \(r\in S\) and \(m\ge 0\). Using this identity and multiplying the above equation by \(A\) on the left, we obtain
		\[ A^{m+1}+(Ar_{m-1})A^{m-1}+\cdots+(Ar_1)A+Ar_0=0,\]
		which shows that \(A\) is integral over \(\Mat_d(I)\) because \(Ar_i\in \Mat_d(I)\). The same computation works for each column, using its diagonal entry. This completes the proof. 
	\end{proof}

	\subsection{Some classes of left KI-rings}\label{sec:positive}
	In this part, we present some classes of left KI-rings. 
	\begin{theorem}\label{thm:positive}
		Let \(R\) be a ring satisfying one of the following properties: 
		\begin{enumerate}[label=(\roman*)]
			\item \(R\) is a left artinian ring.
			\item \(R\) is an algebraic algebra over a field \(k\).
			\item \(R\) is an algebra over a field \(k\) such that \(\dim_kR<|k|\).
		\end{enumerate}
		Then \(R\) is a left KI-ring. 
	\end{theorem}
	\begin{proof}
		In these cases, \(\irad{I}\) coincides with \(\jrad{I}\), which is integral over \(I\) for every left ideal \(I\) by Theorem~\ref{thm:rad-artinian}, Corollary~\ref{cor:k-algebraic}, and Theorem~\ref{thm:amitsur-thm}, respectively. The result now follows from Proposition~\ref{prp:gen-kothe-1st-equiv}. 
	\end{proof}
	It is clear that every commutative ring is a left (and right) KI-ring. More generally, every PI-ring is a left (and right) KI-ring, as shown below.  
	\begin{theorem}\label{thm:PI}
		Let \(R\) be a PI-ring. Then \(R\) is a left (and right) KI-ring. 
	\end{theorem}
	\begin{proof}
		By Proposition~\ref{prp:gen-kothe-1st-equiv}, we only need to show that every G-left ideal is integrally closed. Let \(P\) be a G-left ideal of \(R\) and let 
		\(C\) be its core, that is, the largest two-sided ideal contained in \(P\). It is easy to verify that primeness of \(P\) makes \(C\) prime. In light of Proposition~\ref{prp:int-cls-quotient-ring}, we only need to show that \(P/C\) is integrally closed in \(R/C\). Passing to \(R/C\), we may suppose \(R\) is prime and \(P\) has zero core. Put \(T=Z(R)\setminus\{0\}\), where \(Z(R)\) is the center of \(R\). Since \(P\) is prime with zero core and \(T\) is contained in the center, we see that \(P\cap T=\varnothing\) and \(P\) is \(t\)-saturated for all \(t\in T\).
		
		By Posner's theorem, \(Q=T^{-1}R\) is simple artinian \cite[Theorem 6.5]{drensky2012polynomial}. Let \(Rb\), where \(b\in R\), be integral over \(P\). Every element of \(Qb\) has the form \(s^{-1}rb\) for some \(s\in T\) and \(r\in R\). Then an integral equation for \(rb\) over \(P\), multiplied by the appropriate power of the central unit \(s^{-1}\), proves that \(Qb\) is integral over the left ideal \(T^{-1}P\). Every left ideal of a simple artinian ring is integrally closed by Example~\ref{exm:matrix-int}. Thus \(b/1\in T^{-1}P\), and \(T\)-saturation gives \(b\in P\). This shows that \(P\) is integrally closed, completing the proof.
	\end{proof}
	\begin{remark}
		It is easy to see that each of the above classes of rings satisfies the first condition in Proposition~\ref{prp:gen-kothe-2nd-equiv}; that is, if \(R\) belongs to one of these classes, then every matrix ring \(\Mat_d(R)\) also belongs to the same class. Consequently, these classes satisfy all the equivalent properties listed in Proposition~\ref{prp:gen-kothe-1st-equiv} and Proposition~\ref{prp:gen-kothe-3rd-equiv}.  
	\end{remark}
	We conclude the paper with a conjecture, generalizing the well-known fact that the K\"{o}the conjecture holds for left noetherian rings. 
	\begin{conjecture*}
		Every left noetherian ring is a left KI-ring.
	\end{conjecture*}

	\bibliographystyle{plain}
	\bibliography{references}

\end{document}